\documentclass[12pt,a4paper]{article}
\usepackage{fullpage}
\usepackage{authblk}

\usepackage{mathrsfs}
\usepackage{amsfonts}
\usepackage{amssymb}
\usepackage{bm}
\usepackage{enumerate}
\usepackage{tikz}
\usepackage{graphicx}
\usepackage{stfloats}
\usepackage{amsmath}
\usepackage[colorlinks,
            linkcolor=red,
            anchorcolor=blue,
            citecolor=blue
            ]{hyperref}

\newenvironment{proof}{\noindent{\em \textbf{Proof.}}}{\quad \hfill$\Box$\vspace{2ex}}

\newtheorem{theorem}{Theorem}[section]
\newtheorem{definition}[theorem]{Definition}
\newtheorem{remark}[theorem]{Remark}

\newtheorem{method}{\bf Method}
\numberwithin{equation}{section}

\def\i {\mathbf{i}}

\def \bme {\mathbf{e}}

\newcommand{\abs}[1]{\left|#1\right|}

\title{Monogenic Signal Associated with Linear Canonical Transform and  Application to Edge Detection Problems}

\author[a]{\small  Dong Cheng\thanks{chengdong720@163.com}}

\author[b]{\small Kit Ian Kou\thanks{kikou@umac.mo}}

\affil[a]{\small{Department of Mathematics, Faculty of Arts and Sciences, Beijing Normal University, Zhuhai 519087, China}}
\affil[b]{\small{Department of Mathematics, Faculty of Science and Technology, University of Macau, Macao, China}}

\date{}

\begin{document}
  \maketitle
\begin{abstract}
\normalsize
Monogenic signal is regarded as a generalization of analytic signal from the one dimensional space to the high  dimensional space.
It is defined by an original signal with the combination of Riesz transform. Then it provides the signal features representation, such as the local attenuation and the local phase vector.
The main objective of this study is to analyze the local phase vector and the local attenuation in the high  dimensional spaces. The differential phase congruency is applied for the edge detection problems.
\end{abstract}


\begin{msc}
15A67, 42B10, 30G35
\end{msc}

\section{Analytic signal associated with LCT in one dimensional case}

 Let
$ \{a,b,c,d\}$ be  real parameters satisfying $ad-bc=1$. The linear canonical transform (LCT) \cite{moshinsky1971linear,healy2016linear} of the integrable signal $f$ is defined by
\begin{equation*}
  F^{(a,b,c,d)}(\omega)=\mathscr{L}^{(a,b,c,d)}(f)(\omega):=
 \begin{cases}
 \int_{\mathbf{R}}\mathcal{K}^{(a,b,c,d)}(\omega,x)f(x)dx,& b\neq 0;\\
\sqrt{d}e^{\frac{\i cd}{2}\omega^2}f(d\omega), &b=0;
\end{cases}
\end{equation*}
where
\begin{equation*}
   \mathcal{K}^{(a,b,c,d)}(\omega,x):=\frac{1}{\sqrt{\i2\pi b}}e^{\i(\frac{d}{2b}\omega^2-\frac{1}{b}\omega x+\frac{a}{2b}x^2)}.
\end{equation*}
Note that when $b=0$, the LCT is just a chirp multiplication.  Given the similarity between  the   $b>0$ case and and $b<0$ case, without loss of generality, we always suppose $b>0$ in the following sections.

\begin{remark}
 To   eliminate  the ambiguity of the root sign, we specify $\sqrt{z}=\abs{z}^{\frac{1}{2}}e^{\frac{\i\arg{z}}{2}}$ for every nonzero complex number $z$,  where $\arg{z}$ is the principal argument of $z$ on the  interval $(-\pi,\pi]$.
\end{remark}
If $F^{(a,b,c,d)}$ is also integrable,  then the inversion LCT formula holds, that is
\begin{equation*}
  f(x)=\int_{\mathbf{R}}\mathcal{K}^{(d,-b,-c,a)}(x,\omega)F^{(a,b,c,d)}(\omega)d\omega,~~~~\mathrm{a.e.}
\end{equation*}

From the well-known Plancherel theorem, we know that the LCT can be  extended to $L^2(\mathbf{R})$, the set of square integrable signals. If $f,g \in L^2(\mathbf{R})$, the Parseval's identity gives
\begin{equation*}
  \int_{\mathbf{R}}\abs{f(x)}^2dx=\int_{\mathbf{R}}\abs{F^{(a,b,c,d)}(\omega)}^2d\omega.
\end{equation*}

The LCT has been found wide applications in  signal processing \cite{shinde2011two,shi2012sampling}. On the one hand, the LCT is a  generalization
of many famous linear integral transforms, such as Fourier transform, fractional Fourier transform and Fresnel transform, etc. On the other hand, it is more  flexible for its
extra three degrees of freedom, without increasing the complexity
of the computation (same as  conventional FT).

The parameter $(a,b)$-Hilbert transform (PHT) \cite{fu2008generalized} of   $f$ is defined by
\begin{equation*}
   \mathscr{H}^{(a,b)}(f)(x):=\frac{e^{-\i\frac{a }{2b}x^2}}{\pi}\mathrm{p.v.}\int_{\mathbf{R}}\frac{f(t)e^{\i\frac{a }{2b}t^2}}{x-t}dt=\frac{e^{-\i\frac{a }{2b}x^2}}{\pi} \lim_{\epsilon\to 0^+}\int_{\abs{x-t}>\epsilon}\frac{f(t)e^{\i\frac{a }{2b}t^2}}{x-t}dt,
\end{equation*}
provided the integral exists as a principal value (p.v. means the Cauchy principal value).

It had been shown that   $\mathscr{H}^{(a,b)}$ is an isometry on $L^2(\mathbf{R})$, that is,
 $\mathscr{H}^{(a,b)}$  is a    bijection on $L^2(\mathbf{R})$ and
\begin{equation*}
  \int_{\mathbf{R}}\abs{f(t)}^2dt=\int_{\mathbf{R}}\abs{ \mathscr{H}^{(a,b)}(f)(x)}^2dx
\end{equation*}
for every $f\in  L^2(\mathbf{R})$.

The generalized analytic signal (GAS) \cite{fu2008generalized} associated with LCT is defined as

\begin{equation*}
   f_{\mathscr{A}}^{(a,b)}(x):=f(x)+\i \mathscr{H}^{(a,b)}(f)(x).
\end{equation*}

In fact, the GAS $f_{\mathscr{A}}^{(a,b)}(x)$ can suppress the negative frequency components of $F^{(a,b,c,d)}(\omega)$ in the LCT domain. Therefore if $F^{(a,b,c,d)}\in L^1(\mathbf{R})$, then $f_{\mathscr{A}}^{(a,b)}(x)$ can be expressed as
\begin{equation*}
   f_{\mathscr{A}}^{(a,b)}(x)=2 \int_{0}^{\infty}\mathcal{K}^{(d,-b,-c,a)}(x,\omega)F^{(a,b,c,d)}(\omega)d\omega,~~~~\mathrm{a.e.}
\end{equation*}

In this section, we will show that $f_{\mathscr{A}}^{(a,b)}(x)$ can be extended to a function which is  holomorphic  on the upper half plane. Define
\begin{equation}\label{GAF}
   f_{\mathscr{A}}^{(a,b)}(z):=2 \int_{0}^{\infty}\mathcal{K}^{(d,-b,-c,a)}(z,\omega)F^{(a,b,c,d)}(\omega)d\omega,
\end{equation}
where $z\in \mathbf{C} ^+ :=\{ x+\i y : x \in \mathbf{R},~ y>0\}$. Then we have the following theorem.

\begin{theorem}\label{Titchmash-LCT}
   If $f\in  L^2(\mathbf{R})$ and its  LCT $F^{(a,b,c,d)}\in L^1(\mathbf{R})$. Let $g(t)=e^{ \i\frac{a }{2b}t^2}f(t)$. Then
    $ f_{\mathscr{A}}^{(a,b)}(z)$ defined in (\ref{GAF}) is  holomorphic  on the upper half plane. Moreover,   $ f_{\mathscr{A}}^{(a,b)}(z)$ has the following representation:
    \begin{equation*}
      f_{\mathscr{A}}^{(a,b)}(x+\i y)=e^{-\i\frac{a }{2b}(x+\i y)^2}\left[g(x+\i y)+\i h(x+\i y)\right ]
    \end{equation*}
    where
    \begin{equation}\label{Poisson-integral}
      g(x+\i y)=\int_{\mathbf{R}}P_y(x-t)g(t)dt
    \end{equation}
    and
    \begin{equation}\label{Conj-Poisson-integral}
      h(x+\i y)=\int_{\mathbf{R}}Q_y(x-t)g(t)dt
    \end{equation}
    where $P_y(x)=\frac{1}{\pi}\frac{y}{x^2+y^2}$ and $Q_y(x)=\frac{1}{\pi}\frac{x}{x^2+y^2}$ are Poisson and conjugate Poisson kernel respectively.
\end{theorem}
\begin{proof}
 At first, we show that $F^{(a,b,c,d)}(b\omega)=\frac{e^{\i\frac{bd}{2}\omega^2}}{\sqrt{b}}G^{(0,1,-1,0)}(\omega)$. From the  definition of LCT, we see that
 \begin{equation*}
   \begin{split}
      e^{-\i\frac{bd}{2}\omega^2}F^{(a,b,c,d)}(b\omega) & = e^{-\i\frac{bd}{2}\omega^2}\int_{\mathbf{R}}\frac{1}{\sqrt{\i2\pi b}} e^{\i(\frac{d}{2b}b^2\omega^2-\frac{1}{b}b\omega x+\frac{a}{2b}x^2)}f(x)dx\\
       & =\int_{\mathbf{R}}\frac{1}{\sqrt{\i2\pi b}} e^{-\i\omega x+\frac{a}{2b}x^2}f(x)dx\\
       &=\frac{1}{\sqrt{b}}\int_{\mathbf{R}}\frac{1}{\sqrt{\i2\pi }} e^{-\i\omega x}e^{-\i\frac{a}{2b}x^2}f(x)dx\\
       &=\frac{1}{\sqrt{b}}\int_{\mathbf{R}}\mathcal{K}^{(0,1,-1,0)}(\omega,x)g(x)dx=\frac{1}{\sqrt{b}}G^{(0,1,-1,0)}(\omega).
   \end{split}
 \end{equation*}
By using proper variable substitution, we have

\begin{equation}\label{inverse-representation}
  f_{\mathscr{A}}^{(a,b)}(z) =2b \int_{0}^{\infty}\mathcal{K}^{(d,-b,-c,a)}(z,b\omega)F^{(a,b,c,d)}(b\omega)d\omega.
\end{equation}
Note that $\mathcal{K}^{(d,-b,-c,a)}(z,b\omega)=\frac{1}{\sqrt{b}}e^{-\i\frac{a}{2b}z^2}\mathcal{K}^{(0,-1,1,0)}(z,\omega)e^{-\i\frac{bd}{2}\omega^2}$
 and plug $\frac{e^{\i\frac{bd}{2}\omega^2}}{\sqrt{b}}G^{(0,1,-1,0)}(\omega)$ into (\ref{inverse-representation}), we obtain
 \begin{equation*}
    f_{\mathscr{A}}^{(a,b)}(z) =2 e^{-\i\frac{a}{2b}z^2} \int_{0}^{\infty}\mathcal{K}^{(0,-1,1,0)}(z,\omega)G^{(0,1,-1,0)}(\omega)d\omega.
 \end{equation*}
 Since $\mathcal{K}^{(0,-1,1,0)}$ is the inverse Fourier kernel and $G^{(0,1,-1,0)}$ is the Fourier transform of $g$. Thus, by Titchmarsh's Theorem, we conclude that
 $2\int_{0}^{\infty}\mathcal{K}^{(0,-1,1,0)}(z,\omega)G^{(0,1,-1,0)}(\omega)d\omega$ belongs to $H^2(\mathbf{C}^+)$  and it equals to
 $g(x+\i y)+\i h(x+\i y)$, where $g$ and $h$ are given by (\ref{Poisson-integral}) and (\ref{Conj-Poisson-integral}) respectively. We note that $e^{-\i\frac{a}{2b}z^2}$ is a entire function and therefore    $ f_{\mathscr{A}}^{(a,b)}(z)$ is  holomorphic  on the upper half plane. The proof is complete.
 \end{proof}

From above discussion, we see  that a finite energy signal defined on real line  $\mathbf{R}$  can be extended to an analytic function defined on $\mathbf{C} ^+ $ (upper half complex plane).
 We all know that every  analytic function has many good properties. It is very useful to study signals. More importantly, this  analytic function contains all the information of original signal.  Note that the extension  has two adjustable parameters and therefore it offers some degrees of freedom to the users. We also derive  a convolution representation of the analytic function, it is more convenient to calculate.  To apply this method to image processing, we are going  to study higher dimensional generalization    on Clifford algebra.

\section{Monogenic signal associated with LCT in multidimensional case}

Let $\bme_1,\bme_2,...,\bme_n$ be basic elements satisfying $\bme_i\bme_j+\bme_j\bme_i=-2\delta_{ij}$, where $\delta_{ij}$ is the Kronecker delta function, that is, $\delta_{ij}=1$ if $i=j$ and  $\delta_{ij}=0$ otherwise, $i,j=1,2,...,n$. The real (complex) Clifford algebra \cite{brackx1982clifford}, denoted by $\mathbf{R}^{(n)}$ ($\mathbf{C}^{(n)}$), is the associative algebra generated by $\bme_1,\bme_2,...,\bme_n$, over the real (complex) field $\mathbf{R}$ ($\mathbf{C}$). A general element in $\mathbf{R}^{(n)}$ ($\mathbf{C}^{(n)}$) is of the form $x=\sum_S x_S\bme_S$ where $x_S\in \mathbf{R}$ ($\mathbf{C}$), $\bme_S=\bme_{i_1}\bme_{i_2}\cdots\bme_{i_l}$, and $S$ runs over all the ordered subsets of $\{1,2,...,n\}$, namely $S\in\bigcup_{l=0}^n \Lambda_l$ with $\Lambda_l=\{\{i_1,i_2,\cdots,i_l\}: 1\leq i_1<i_2\cdots<i_l\leq n\}$.   $\Lambda_0$ only contains empty set, for $S=\emptyset$, we write $\bme_S=\bme_0=1$. In particular, $\sum_{S\in\Lambda_0} x_S\bme_S$ and $\sum_{S\in\Lambda_1} x_S\bme_S$ will be denoted by
$\mathrm{Sc}(x)$ and $\mathrm{Vec}(x)$ respectively. They are called the scalar part and  vector part of   (real or complex) Clifford number $x$, respectively.
 Let
\begin{equation*}
   \mathbf{R}^n=\{\underline{x}=x_1\bme_1+x_2\bme_2+\cdots+x_n\bme_n:x_j\in\mathbf{R},j=1,2,...,n\}
\end{equation*}
be   identical with the usual Euclidean space $\mathbf{R}^n$, and
 \begin{equation*}
   \mathbf{R}^n_1=\{x_0+\underline{x}: x_0\in \mathbf{R}, \underline{x}\in \mathbf{R}^n\}.
\end{equation*}
An element in $\mathbf{R}^n_1$ is called a para-vector. The natural inner product between $x=\sum_Sx_S\bme_S$ and $y=\sum_S y_S\bme_S$ in $\mathbf{C}^{(n)}$ is defined by $\langle x,y\rangle=\sum_Sx_Sy_S$. The norm associated with this inner product is $\abs{x}=\langle x,x\rangle^{\frac{1}{2}}=\left(\sum_S\abs{x_S}^2\right)^\frac{1}{2}$. The conjugate of a para-vector $x=x_0+\underline{x}\in \mathbf{R}^n_1$ is defined as $\overline{x}=x_0-\underline{x}$.

Let $n=3$, then
\begin{equation}\label{clifford number}
\begin{split}
    x= &  1+\i\\
    & +(2+5\i)\bme_1+(1+2\i)\bme_2+(3+\i)\bme_3\\
    &  +(2+6\i)\bme_1\bme_2+(5+3\i)\bme_1\bme_3 +(1+\i)\bme_2\bme_3 \\
    &+(6+9\i)\bme_1\bme_2\bme_3
\end{split}
\end{equation}
is a complex Clifford number (in $\mathbf{C}^{(3)}$). The scalar part of $x$ is
\begin{equation*}
 \mathrm{Sc}(x)=1+\i.
\end{equation*}
The vector part of $x$ is
\begin{equation*}
  \mathrm{Vec}(x)=(2+5\i)\bme_1+(1+2\i)\bme_2+(3+\i)\bme_3.
\end{equation*}

Let $\Omega$ be an open subset of $\mathbf{R}^n_1$  with a piecewise smooth boundary. The function defined on  $\Omega$ taking values in  $\mathbf{C}^{(n)}$ has the form $f(x_0+\underline{x})=\sum_Sf_S(x_0+\underline{x})\bme_S$, where $f_S$ are complex-valued functions. The Dirac operator, a generalization of Cauchy-Riemann operator, is defined by
\begin{equation*}
  D=\frac{\partial}{\partial x_0}+\underline{D}= \frac{\partial}{\partial x_0}+\frac{\partial}{\partial x_1}\bme_1+\frac{\partial}{\partial x_2}\bme_2+\cdots+\frac{\partial}{\partial x_n}\bme_n.
\end{equation*}
 Since $\mathbf{C}^{(n)}$   is not a   commutative algebra, in general   $Df$ is not same as $fD$, they are respectively given by
 \begin{equation*}
  Df=\sum_{j=0}^n\sum_S \frac{\partial f_S}{\partial x_j}\bme_j\bme_S
 \end{equation*}
 and
 \begin{equation*}
  fD=\sum_{j=0}^n\sum_S \frac{\partial f_S}{\partial x_j}\bme_S\bme_j
 \end{equation*}
A  function $f$ is said to be left (resp. right) monogenic in $\Omega$ if $Df=0$ (resp. $fD=0$) in $\Omega$.

\begin{remark}
   In the following sections, we will only use left monogenic functions. Hence we will simply say monogenic for left monogenic.
\end{remark}

Let
\begin{equation*}
  \mathbf{C}_{v}:=\{y=\sum_{j=0}^n y_j\bme_j:   0\leq j\leq n, y_j\in \mathbf{C} \}.
\end{equation*}
be a subset of $\mathbf{C}^{(n)}$.  That is, any element in $\mathbf{C}_{v}$  only consists of scalar part and vector part. If $\sqrt{y_0^2+y_1^2+\cdots+y_n^2}\sqrt{y_1^2+y_2^2+\cdots+y_n^2}\neq 0$, then
\begin{equation*}
  \begin{split}
    y= & y_0+y_1\bme_1+\cdots+y_n\bme_n \\
     = &  \sqrt{y_0^2+ \cdots+y_n^2}\left( \frac{y_0}{ \sqrt{y_0^2 +\cdots+y_n^2} }  +\frac{y_1\bme_1 +\cdots+y_n\bme_n}{\sqrt{y_1^2 +\cdots+y_n^2}}\frac{\sqrt{y_1^2 +\cdots+y_n^2}}{\sqrt{y_0^2+ \cdots+y_n^2}} \right)\\
     =&A  \left( \cos\theta +\mathbf{I}\sin\theta \right)\\
     =&A e^{\mathbf{I}\theta}
  \end{split}
\end{equation*}
is the polar form of $y\in \mathbf{C}_{v}$. Note that $A$ and $\theta$ are complex-valued, we shall call them   complex 'norm' and complex 'phase' respectively. The complex 'phase' is given by
\begin{equation*}
  \theta=  \begin{cases}
 \frac{\pi}{2},& y_0= 0;\\
 \arctan\frac{\sqrt{y_1^2 +\cdots+y_n^2}}{y_0},  &y_0\neq 0;
\end{cases}
\end{equation*}
where $\arctan z=\frac{1}{2\i} \ln\frac{1+\i z}{1-\i z}$ and $\ln z=\ln\abs{z}+\i \arg z$ with $-\pi<\arg z\leq \pi$.

We denote the integrable and square integrable  $\mathbf{C}^{(n)}$-valued functions (defined on $\Omega \in \mathbf{R}^n$) by $L^1(\Omega,\mathbf{C}^{(n)})$ and  $L^2(\Omega,\mathbf{C}^{(n)})$ respectively. If $f\in L^1(\mathbf{R}^n,\mathbf{C}^{(n)})$, the LCT of $f$ is defined by
\begin{equation*}
  F^{(a,b,c,d)}(\underline{\xi}) :=
 \int_{\mathbf{R}^n}\mathcal{K}^{(a,b,c,d)}( \underline{\xi},\underline{x})f(\underline{x})d\underline{x},
\end{equation*}
where
\begin{equation*}
   \mathcal{K}^{(a,b,c,d)}( \underline{\xi},\underline{x}):=\frac{1}{\sqrt{\i2\pi b}}e^{\i(\frac{d}{2b} \abs{\underline{\xi}}^2-\frac{1}{b}   \langle \underline{\xi}, \underline{x}\rangle +\frac{a}{2b}\abs{\underline{x}}^2)}.
\end{equation*}

If in addition, $F^{(a,b,c,d)}\in L^1(\mathbf{R}^n,\mathbf{C}^{(n)})$,  then the inversion LCT formula holds, that is
\begin{equation*}
  f(\underline{x})=\int_{\mathbf{R}^n}\mathcal{K}^{(d,-b,-c,a)}(\underline{x},\underline{\xi})F^{(a,b,c,d)}(\underline{\xi})d\underline{\xi},~~~~\mathrm{a.e.}
\end{equation*}

Like the one dimensional case, the LCT can   be  extended to $L^2(\mathbf{R}^n,\mathbf{C}^{(n)})$ as well.

To extend the notion of analytic signal to higher dimensional, we give the generalized monogenic signal as follows:

\begin{definition}[Generalized monogenic signal]
  For $f\in  L^2(\mathbf{R}^n,\mathbf{C}^{(n)})$, the generalized monogenic signal is defined by
  \begin{equation*}
 f_{\mathscr{M}}^{(a,b)}(\underline{x}):=f(\underline{x})-\sum_{j=1}^n\mathscr{R}_j^{(a,b)}(f)(\underline{x})\bme_j,
\end{equation*}
where
\begin{equation*}
 \mathscr{R}_j^{(a,b)}(f)(\underline{x}):=\frac{\Gamma(\frac{n+1}{2})}{\pi^{\frac{n+1}{2}}}\mathrm{p.v.}
 \int_{\mathbf{R}^n}\frac{e^{-\i\frac{a}{2b}\abs{\underline{x}}^2}(x_j-t_j)f(\underline{t})e^{\i\frac{a}{2b}\abs{\underline{t}}^2}}{\abs{\underline{x}-\underline{t}}^{n+1}}d\underline{t},
\end{equation*}
is the $j$-th generalized Reisz transform of $f$.
\end{definition}
\begin{remark}
  By definition, if $f$ is real-valued, then $f_{\mathscr{M}}^{(a,b)}$ is $\mathbf{C}_{v}$-valued.
\end{remark}

Applying  the LCT to $ f_{\mathscr{M}}^{(a,b)}$, we have
\begin{equation*}
  \mathscr{L}^{(a,b,c,d)}(f_{\mathscr{M}}^{(a,b)})(\underline{\xi})=F^{(a,b,c,d)}(\underline{\xi}) \left(1+\i \frac{\underline{\xi}}{\abs{\underline{\xi}}}\right).
\end{equation*}
Therefore if $F^{(a,b,c,d)}\in   L^1(\mathbf{R}^n,\mathbf{C}^{(n)})\cap L^2(\mathbf{R}^n,\mathbf{C}^{(n)})$, the inversion LCT formula gives
\begin{equation*}
   f_{\mathscr{M}}^{(a,b)}(\underline{x})= \int_{\mathbf{R}^n}\mathcal{K}^{(d,-b,-c,a)}(\underline{x},\underline{\xi})(1+\i \frac{\underline{\xi}}{\abs{\underline{\xi}}})
   F^{(a,b,c,d)}(\underline{\xi})d\underline{\xi},~~~~\mathrm{a.e.}
\end{equation*}

We define
\begin{equation*}
   f_{\mathscr{M}}^{(a,b)}(x_0+\underline{x}):= \frac{e^{-\i\frac{a}{2b}\abs{x_0+\underline{x}}^2}  }{\sqrt{-\i 2\pi b}} \int_{\mathbf{R}^n} e^{-\frac{x_0}{b}\abs{\underline{\xi}}}e^{\i\frac{1}{b}<\underline{x},\underline{\xi}>}(1+\i \frac{\underline{\xi}}{\abs{\underline{\xi}}})
   F^{(a,b,c,d)}(\underline{\xi})e^{-\i\frac{d}{2b}\abs{\underline{\xi}}^2}d\underline{\xi},
\end{equation*}
in  $\mathbf{R}_1^{n,+}:=\{(x_0,\underline{x}):x_0>0,\underline{x}\in \mathbf{R}^n\}$.

By some similar argument to Theorem  \ref{Titchmash-LCT}, we have following result.

\begin{theorem}\label{Titchmash-LCT-clifford}
   If $f\in   L^2(\mathbf{R}^n,\mathbf{C}^{(n)})$ and its  LCT $F^{(a,b,c,d)}\in L^1 (\mathbf{R}^n,\mathbf{C}^{(n)})$. Let $g(\underline{t})=e^{ \i\frac{a }{2b}\abs{\underline{t}}^2}f(\underline{t})$. Then
    $  f_{\mathscr{M}}^{(a,b)}(x_0+\underline{x})$   has the following representation:
    \begin{equation*}
     f_{\mathscr{M}}^{(a,b)}(x_0+\underline{x})=e^{-\i\frac{a}{2b}\abs{x_0+\underline{x}}^2}
     \left[g(x_0+ \underline{x})+  h(x_0+ \underline{x})\right ]
    \end{equation*}
    where
    \begin{equation}\label{Poisson-integral-clifford}
      g(x_0+ \underline{x})=\int_{\mathbf{R}^n}P_{x_0}(\underline{x}-\underline{t})g(\underline{t})d\underline{t}
    \end{equation}
    and
    \begin{equation}\label{Conj-Poisson-integral-clifford}
      h(x_0+ \underline{x})
      =\int_{\mathbf{R}^n}Q_{x_0}(\underline{x}-\underline{t})g(\underline{t})d\underline{t}
    \end{equation}
    where $P_{x_0}(\underline{x})=\frac{\Gamma(\frac{n+1}{2})}{\pi^{\frac{n+1}{2}}}\frac{x_0}{\abs{x_0+\underline{x}}^{n+1}}$ and $Q_{x_0}(\underline{x})=\frac{\Gamma(\frac{n+1}{2})}{\pi^{\frac{n+1}{2}}}
    \frac{\underline{\overline{x}}}{\abs{x_0+\underline{x}}^{n+1}}$ are Poisson and conjugate Poisson kernel respectively.
\end{theorem}
\begin{remark}
In the complex case,  $f_{\mathscr{A}}^{(a,b)}(z)$ is analytic. In the Clifford case, however, $f_{\mathscr{M}}^{(a,b)}(x_0+\underline{x})$ is not monogenic.  Nevertheless, it is "almost" monogenic, because $$e^{\i\frac{a}{2b}\abs{x_0+\underline{x}}^2}f_{\mathscr{M}}^{(a,b)}(x_0+\underline{x})$$ is monogenic in $\mathbf{R}_1^{n,+}$.
\end{remark}

Note that $f_{\mathscr{M}}^{(a,b)}(x_0+\underline{x})$   is $\mathbf{C}_{v}$-valued.  Let
\begin{equation*}
  f_{\mathscr{M}}^{(a,b)}=f_0+f_1\bme_1+f_2\bme_2+\cdots+f_n\bme_n,
\end{equation*}
then  it has the following polar form:
\begin{equation}\label{polar-form-of-monogenic-signal}
 f_{\mathscr{M}}^{(a,b)}(x_0+\underline{x})=A(x_0+ \underline{x})e^{\mathbf{I}(x_0+\underline{x})\theta(x_0+\underline{x})}
\end{equation}
where
\begin{equation*}
A =\sqrt{f_0^2+f_1^2+\cdots+f_n^2}
\end{equation*}
is called the local complex amplitude,
\begin{equation*}
  \theta=\arctan\frac{\sqrt{f_1^2 +f_2^2\cdots+f_n^2}}{f_0}
\end{equation*}
is called the complex phase,
\begin{equation*}
  r=\mathbf{I}\theta
\end{equation*}
with $\mathbf{I}=\frac{f_1\bme_1+f_2\bme_2 +\cdots+f_n\bme_n}{\sqrt{f_1^2 +\cdots+f_n^2}}$ is called the local complex phase vector.

\begin{remark}
   In general, the newly defined $A$ and $\theta$ are complex-valued. By choosing specific parameters $a,b$  ($a=0, b=1$ for example ), $A$, $\theta$ reduce to be real-valued, and with   range  $[0,\infty)$, $(-\frac{\pi}{2},\frac{\pi}{2}]$ respectively.
\end{remark}

Suppose that the local complex amplitude never be zero, we call $\rho=\ln A$ the local complex attenuation. Then we can write (\ref{polar-form-of-monogenic-signal}) as
\begin{equation}\label{polar-form-of-monogenic-signal-new}
 f_{\mathscr{M}}^{(a,b)}(x_0+\underline{x})=e^{\rho(x_0+\underline{x})}e^{\mathbf{I}(x_0+\underline{x})\theta(x_0+\underline{x})}
\end{equation}

Note that $e^{\i\frac{a}{2b}\abs{x_0+\underline{x}}^2}f_{\mathscr{M}}^{(a,b)}(x_0+\underline{x})$ is monogenic in $\mathbf{R}_1^{n,+}$. Applying Dirac operator to this function, we have
\begin{equation*}
  (\frac{\partial}{\partial x_0}+\underline{D})e^{\i\frac{a}{2b}\abs{x_0+\underline{x}}^2}
  e^{\rho(x_0+\underline{x})}e^{r(x_0+\underline{x})}=0.
\end{equation*}

By direct computation, we obtain
\begin{align}
    &\i \frac{a}{b}x_0+\frac{\partial \rho}{\partial x_0}+ \mathrm{Sc}[(\underline{D}e^r)e^{-r}] = 0,  \label{CauchyRiemann1} \\
   & \i \frac{a}{b}\underline{x} + \frac{\partial r}{\partial x_0} + \underline{D}\rho -
   \mathrm{Vec}[(\underline{D}\mathbf{I})\mathbf{I}]\sin^2\theta+(\sin\theta\cos\theta-\theta)
   \frac{\partial \mathbf{I}}{\partial x_0}=0. \label{CauchyRiemann2}
\end{align}
In fact, (\ref{CauchyRiemann1}) and (\ref{CauchyRiemann2}) constitute an analogue of Cauchy-Riemann system.

\section{Edge detection}

We aim to apply above theory to image processing. Let $n=2$, consider an image $u(\underline{x})$ defined in $\mathbf{R}^2$. By Theorem  \ref{Titchmash-LCT-clifford}, we can construct a function $f(x_0+\underline{x})$ defined    in $\mathbf{R}_1^{2,+}$. If we write $f(x_0+\underline{x})$ as
\begin{equation*}
\begin{split}
   f(x_0+\underline{x})= & A(x_0+ \underline{x})e^{\mathbf{I}(x_0+\underline{x})\theta(x_0+\underline{x})} \\
    =& e^{\rho(x_0+\underline{x})}e^{r(x_0+\underline{x})}
\end{split}
\end{equation*}
 It is interesting to see that   $f(x_0+\underline{x})$ is a "good" function (because it is "almost" monogenic and satisfies (\ref{CauchyRiemann1}), (\ref{CauchyRiemann2}) ).  More importantly, $f(x_0+\underline{x})$ contains all the information of $u(\underline{x})$.

 We introduce the following methods to detect edges of images.
 \begin{method}[LCA]
 Applying the Dirac operator $\underline{D}$ to the local complex attenuation $\rho$. Computing  the norm of $\underline{D}\rho$ to get the gradient map.
\end{method}

 \begin{method}[MDCPC]
Computing modified differential complex phase congruency (MDCPC) to detect edges, where MDCPC is defined by
\begin{equation*}
 \i \frac{a}{b}\underline{x} + \frac{\partial r}{\partial x_0}   -
   \mathrm{Vec}[(\underline{D}\mathbf{I})\mathbf{I}]\sin^2\theta+(\sin\theta\cos\theta-\theta)
   \frac{\partial \mathbf{I}}{\partial x_0}
\end{equation*}
\end{method}


Method 1 and Method 2  are based on the amplitude information and the phase information of  the monogenic function respectively.  With parameters $\{a,b,c,d\}=\{0,1,-1,0\}$, the LCT reduces to classical Fourier transform (FT). The monogenic signal associated with FT has been introduced by  Yang et al. in \cite{yang2017edge}.  
   They  use local attenuation (LA) and modified differential   phase congruency (MDPC) to detect edges. We generalize    their    results. LA  and MDPC  are generalized to LCA and MDCPC.  In the present study,
   we use three different images for the comparison of LCA (associated with LCT) and LA (associated with FT). 
 We also compare the performance of MDCPC (associated with LCT) and MDPC (associated with FT).

 The experimental results show the superiority of our methods. The
adjustable parameters help us to get better results in edge detection.

\begin{figure}[!htb]
  \centering
  \includegraphics[width=15.3cm]{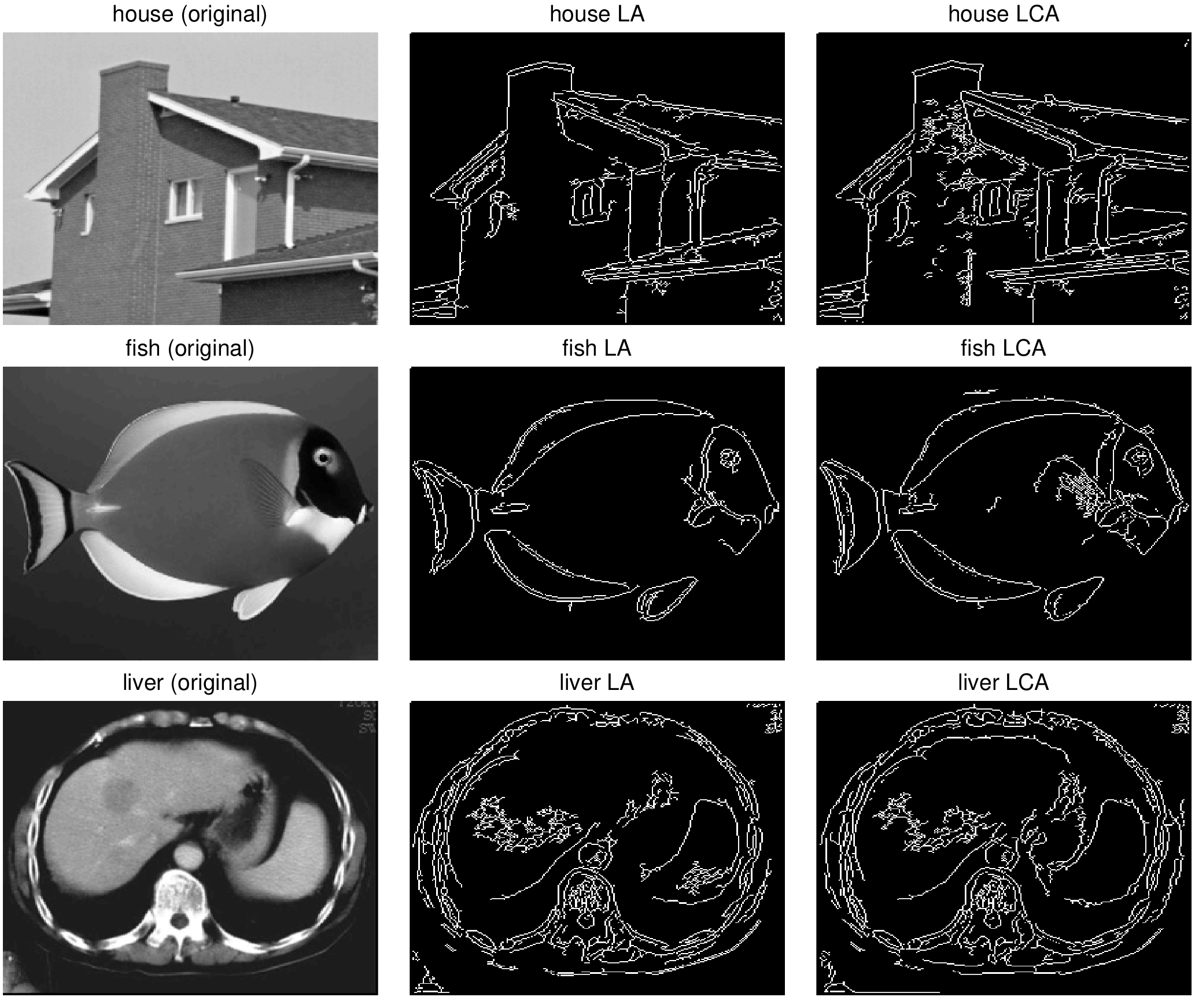}\\
  \caption{Edge detection results by LA and LCA.}\label{1}
\end{figure}

\begin{figure}[!htb]
  \centering
  \includegraphics[width=15.3cm]{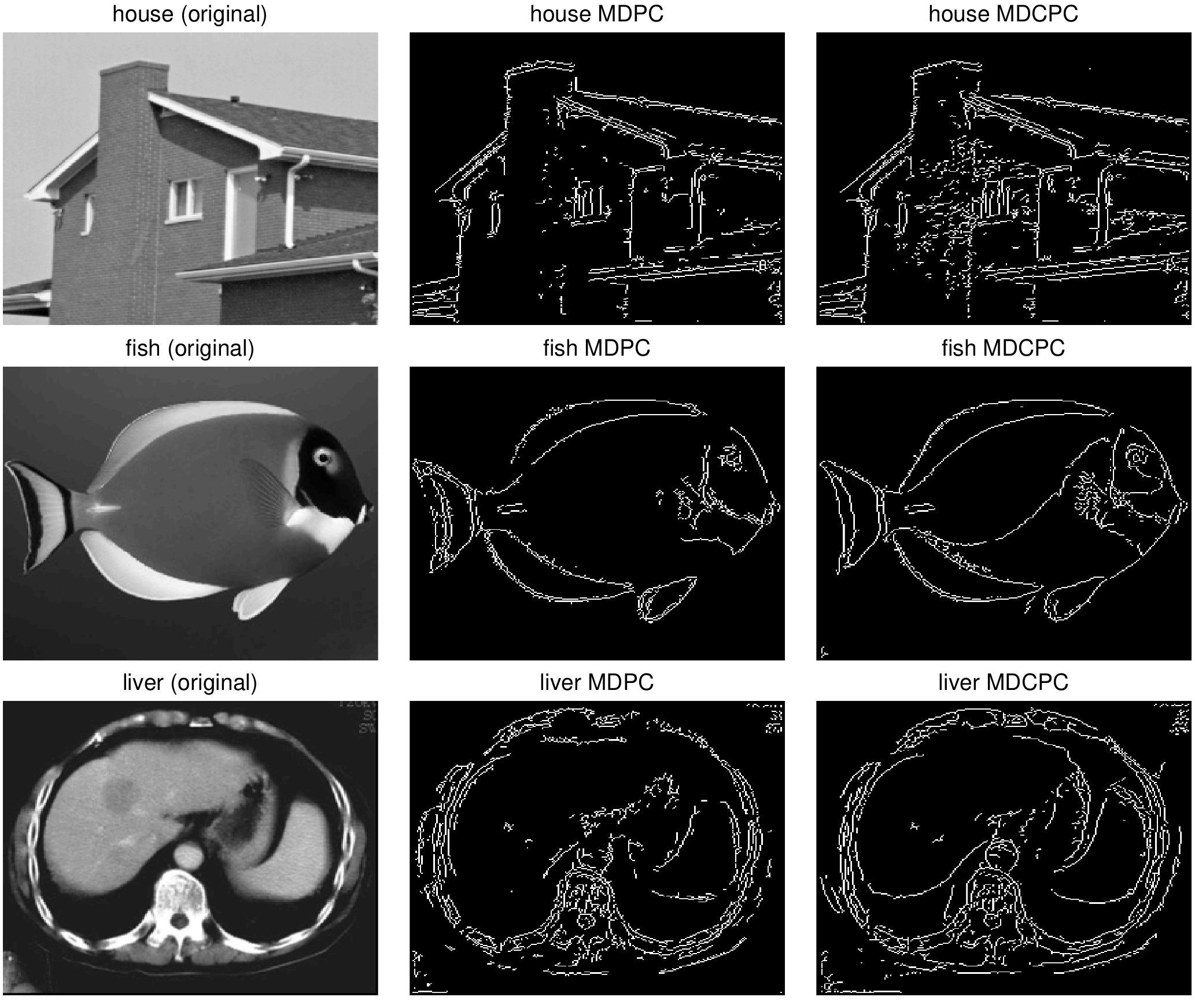}\\
  \caption{Edge detection results by MDPC and MDCPC.}\label{2}
\end{figure}


{
\bibliographystyle{IEEEtran}
\bibliography{IEEEabrv,myreference20220826}
}









\end{document}